\documentclass[a4paper, 10 pt, reqno]{amsproc}

\usepackage{defs}
\usepackage{fontawesome}
\usepackage{orcidlink}

\begin{document}

\author[S. Das]{Souvik Das\orcidlink{0000-0001-6918-6219}} 
\author[S. Ganguly]{Siddhartha Ganguly\orcidlink{0000-0003-2046-2061}} 
\author[A. Aravind]{Ashwin Aravind\orcidlink{0000-0002-6412-5772}} 
\author[D. Chatterjee]{Debasish Chatterjee\orcidlink{0000-0002-1718-653X}}

\title[Data-driven distributionally robust MPC]{Data-driven distributionally robust MPC for systems with multiplicative noise: A semi-infinite semi-definite programming approach}
\thanks{S. Das, S. Ganguly, A. Aravind, and D. Chatterjee are with \faGroup\ Centre for Systems and Control, \faUniversity\ Indian Institute of Technology, Mumbai, 400076, India.  \\
	(SD): \faEnvelope\ \texttt{souvikd@iitb.ac.in}, \faHome\ \url{https://sites.google.com/view/souvikd}.\\
  \noindent (SG): \faEnvelope\ \texttt{sganguly@iitb.ac.in}, \faHome\ \url{https://sites.google.com/view/siddhartha-ganguly}.\\
  \noindent (AA): \faEnvelope\ \texttt{a.aravind@iitb.ac.in}, \faHome\ \url{http://www.sc.iitb.ac.in/~ashwinaravind}.\\
(DC): \faEnvelope\ \texttt{dchatter@iitb.ac.in}, \faHome\ \url{http://www.sc.iitb.ac.in/~chatterjee}.
}
\thanks{This work was partially supported by Ministry of Human Resource Development, Govt. of India. D. Chatterjee acknowledges partial support of the SERB MATRICS grant MTR/2022/000656.}

\thanks{SG thanks Prof. P. Goulart for identifying an error in equation (5) during the review of his Ph.D. thesis.}

\begin{abstract}     
This article introduces a novel distributionally robust model predictive control (DRMPC) algorithm for a specific class of controlled dynamical systems where the disturbance multiplies the state and control variables. These classes of systems arise in mathematical finance, where the paradigm of distributionally robust optimization (DRO) fits perfectly, and this serves as the primary motivation for this work. We recast the optimal control problem (OCP) as a semi-definite program with an infinite number of constraints, making the ensuing optimization problem a \emph{semi-infinite semi-definite program} (\(\sisdp\)). To numerically solve the \(\sisdp\), we advance an approach established in \cite{ref:DasAraCheCha-22} in the context of convex semi-infinite programs (SIPs) to \(\sisdp\)s and subsequently, solve the DRMPC problem. A numerical example is provided to show the effectiveness of the algorithm.
\end{abstract}

\maketitle
%===============================================================================
%% =============================================================================
\section{Introduction}
\label{s:intro}
%% =============================================================================

This article focuses on the technique of model predictive control (MPC) of uncertain stochastic dynamical systems where the uncertainties (we use the terms uncertainties and disturbances interchangeably) multiply the system state and control variables --- these classes of systems arise naturally in applications related to finance such as portfolio optimization \cite{ref:JAP:ACC:portfolio}, constrained index tracking \cite{ref:JAP:SCH:index_tracking}, trading applications \cite{ref:SM:LJH:JAP:finance} etc. MPC is perhaps one of the most popular and practically deployed optimization-based control synthesis technique which has witnessed explosive growth and proliferation in several industries. The theoretical aspects of MPC, such as stability and feasibility, are quite well-developed for deterministic, robust, and stochastic systems \cite{ref:May-16} under several types of constraints, but primarily due to the nature of the synthesis technique, computational tractability remains the primary bottleneck. 

Deterministic or nominal MPC techniques are designed to deliver good performance under constraints in the absence of disturbances. Robust MPC techniques take care of the uncertainties by employing a min-max optimization problem with bounded disturbances and enforcing the state and the control constraints for all possible disturbance realizations; the ensuing optimal control problem becomes a \emph{semi-infinite} program but it applies to safety-critical applications. Stochastic MPC specifies a probability distribution \(\PP_{\dist}\) of the disturbance \(\dist\) and typically solves a chance-constrained stochastic program over a class of policies. While a small margin of distributional robustness is inherent in SMPC techniques~\cite{ref:RDM:JBR:TAC:23}, when the underlying distribution is furnished from a relatively `large' set of data sites, the resulting uncertainty is typically too large to ignore. This serves as a primary motivation to equip the SMPC enterprise with techniques from distributionally robust optimization (DRO). 

DRO techniques \cite{ref:mohajerin2018data} assume that the underlying probability distribution of the uncertain parameters belongs to a set of distributions \(\setofdist\) which we will refer to as the \emph{ambiguity set}. The ensuing optimal control problem is then posed as an \(\inf\)-\(\sup\) constrained optimization, where the \(\sup\) is applied over the ambiguity set \(\setofdist\) with the expected cost under an unknown distribution \(\unknowndist \in \setofdist\). In the context of distributionally robust MPC (DRMPC), among other things, numerical tractability is one of the key difficulties and to address this challenge, several approaches have been reported in the literature. For example, tractable formulations for linear controlled dynamical systems using state-feedback policies of the form \(u = K x + \eta\), disturbance feedback policies of the form \(u = \theta w + \eta\), along with several types of convex reformulations of chance or conditional value-at-risk type of state constraints were introduced to obtain certain convex structures; see \cite{ref:BPGVP:DK:PJG:MM:DRMPC:TAC16, ref:CM:SL:LCSS:22:DD-DRMPC, ref:BL:YT:AW:GD:TAC21:DRSMPC}. The choice of the ambiguity set also plays a crucial role in the tractability and performance of DRMPC algorithms. Several types of ambiguity sets, such as moment ambiguity sets \cite{ref:EDYY_moment_ambiguity:10}, their inner and outer approximations \cite{ref:YT:JY:WHC:SL:TAC24:DRSMPC}, and Wasserstein balls \cite{ref:YI:WassDRMPC:TAC20} have been employed in DRMPC framework, but all in the context of linear systems. We further draw attention to \cite{ref:RDM:PME:DRMPC23} where a collection of closed-loop stability results were established for stochastic linear systems with bounded noise and Gelbrich ambiguity sets. 

In the context of dynamical systems that are not necessarily linear, e.g., where the disturbance multiplies the state and the control variables, data-driven model-based algorithms relying on translating the underlying optimal control to a semi-definite program (SDP) were established in \cite{ref:PC:MS:PP_modelbased:L4DC:20} using conical ambiguity sets. The target classes of problems were unconstrained linear quadratic regulators. Leveraging tools from multi-linear tensor algebra a distributionally robust optimal control strategy was established in \cite{ref:PC:PP:TAC23:DRMPC:mult_noise} for systems with multiplicative noise. 

\noindent \textbf{Our contributions} 
\begin{itemize}[leftmargin=*]
\item In \cite{ref:DasAraCheCha-22} a framework to extract \emph{exact} solutions for convex semi-infinite programs was established. We extend this framework to the context of semi-infinite semi-definite programs (\(\sisdp\)s). Our algorithm guarantees, under mild structural assumptions that the value and the optimizers of the original \(\sisdp\) are the same as those of a suitably relaxed version of the \(\sisdp\); see \S\ref{sec:prob_form}.
\item The centerpiece of our study is a discrete-time stochastic MPC (SMPC) problem where the disturbance multiplies the state and the control variables. We introduce distributional uncertainty over certain ambiguity sets and formulate the given SMPC problem as a DRMPC. Subsequently, we translate the DRMPC into an \(\sisdp\) and apply the results established in \S\ref{sec:prob_form}. We illustrate our results with the aid of a numerical example. 
\end{itemize}

The algorithm reported herein is typically slow due to the presence of a global optimization step and one of our motivations behind this development is the usage of these results along with an explicit MPC oracle along the lines of \cite{ref:GanCha-22}. Along with the explicit MPC algorithm, stability and recursive feasibility guarantees of the online algorithm is under development, and will be reported jointly in a subsequent article.

%% ===========================
\noindent \textbf{Notation:}
%% ===========================
Let \((\ps,\tsigalg,\PP)\) be a (sufficiently rich) probability space, and we assume that all random elements are defined on \((\ps,\tsigalg,\PP)\). The realization of a \(d\mbox{-}\)dimensional random vector \(g\) at \(\omega \in \ps\) defined on \((\ps,\tsigalg,\PP)\) is given by \(g(\omega)\). Let \(f\) be another random variable defined on \((\ps,\tsigalg,\PP)\) taking values in some Euclidean space and \(\PP_{f}\) denotes the distribution of \(f\), i.e., \(\PP_{f}(S) = \PP(f \in S)\) for every \(S \in \tsigalg\). Moreover, we adopt the notation \(\EE_{\ol{g}} \expecof[\big]{f} = \EE \cexpecof[\big]{f\given \ol{g}}\). We further assume that \(\EE \expecof[\big]{\abs{g}^2}\) exists and is finite. In the rest of the article we omit these details for the sake of brevity, and \emph{with a slight abuse of notation and write \(g\) as the realization taking values in \(\Rbb^d\).} 
We let \(\N \Let \aset{1,2,\ldots}\) denote the set of positive integers. The vector space \(\R[d]\) is assumed to be equipped with standard inner product \(\inprod{v}{v'}\Let \sum_{j=1}^d v_j v'_j\) for every \(v,v' \in \R[d]\). For any arbitrary subset \(X \subset \Rbb^d\) we denote the interior of \(X\) by \(\intr X\). We denote the set of all \(n\times n\) matrices with real entries by \(\matr{n}\), the set of all symmetric matrices by \(\symmat{n}\), the set of all symmetric positive definite matrices by \(\pdmat{n}\), and the set of all positive semi-definite matrices by \(\psdmat{n}\). We equip the space \(\symmat{n}\) with the inner product \(\symmat{n} \times \symmat{n} \ni (A,B) \mapsto \inprod{A}{B} \Let \trace(A^{\transpose}B)= \sum_{i,j=1}^{n}a_{ij}b_{ij}\).

%%%%%%%%%%% Problem formulation %%%%%%%%%%%

\section{Problem formulation}\label{sec:prob_form}

Consider the semi-infinite semi-definite program (\(\sisdp\)): 
\begin{equation}
	\label{eq:mtns:CSIP_original}
	\begin{aligned}
		\ov^{\ast}=& \inf_{\dvm \in \dvset} && \inprod{C}{\dvm} \\
		& \sbjto && \begin{cases}
			 \inprod{\uvar}{\dvm} \leqslant b \quad \text{for all}\, \uvar\in \uset,\\
            \dvm \succeq 0,
            \end{cases}
	\end{aligned}
\end{equation}
where the \emph{domain} \(\dvset \Let \symmat{n} \subset \matr{n}\) is a closed and convex set with non-empty interior, the matrices \(C,\uvar \in \symmat{n}\) and \(b \in \Rbb\), and the \emph{constraint index set} $\uset \subset \symmat{n}$ is compact with nonempty interior. The \emph{admissible set} \(\feas\) is defined by 
\[\feas \Let \aset[\big]{\dvm \in \dvset \suchthat \dvm \succeq 0,\, \inprod{\uvar}{X} \leq b \,\, \text{for all }\uvar \in \uset},\] and it is assumed to have a non-empty interior, i.e., there exists a symmetric matrix \(\ol{\dvm} \in \dvset\) such that \(\dvm \succeq 0\) and \(\langle\uvar,\ol{\dvm}\rangle<b\) for every \(\uvar{}{} \in \uset\).\footnote{The assumption implies that the problem \eqref{eq:mtns:CSIP_original} is strictly feasible for every \(\uvar \in \uset\).  Consequently, there exists a symmetric matrix \(\ol{\dvm} \in \dvset\) such that \(\ol{\dvm} \succeq 0\) and for every (for a fixed \(n \in \N\)) \(n\mbox{-}\)tuple \(\bigl(\uvar_1,\uvar_2,\ldots,\uvar_n \bigr) \in \uset^n\), \(\inprod{\uvar}{\ol{\dvm}} < b\) for every \(i=1,2,\ldots,n\).}

The requirement that the decision matrix \(\dvm \in \psdmat{n}\) is equivalent to the condition that \(
\uvarr{\transpose}\dvm \uvarr \ge 0 \, \text{for all }\uvarr \in \Rbb^n, 
\) consequently, the optimization problem \eqref{eq:mtns:CSIP_original} can be written as
\begin{equation}
	\label{eq:CSIP_aux}
	\begin{aligned}
		\ov^{\ast}=& \inf_{\dvm \in \dvset} && \inprod{C}{\dvm} \\
		& \sbjto && \begin{cases}
			 \inprod{\uvar}{\dvm} \leqslant b \quad \text{for all}\, \uvar\in \uset,\\
            \uvarr\transpose\dvm \uvarr \geq 0 \quad \text{for all }\uvarr \in \Rbb^n.
            \end{cases}
	\end{aligned}
\end{equation}
We assume that \(y{\as}\) can take the value $-\infty$. 
The family $\aset[\big]{\inprod{\uvar}{\dvm} \leq b,\, \uvarr\transpose\dvm \uvarr \geq 0 \suchthat  \, \uvar \in \uset, \uvarr \in \Rbb^n}$ of constraints is often known as \emph{semi-infinite constraints} since it may contain uncountably many inequality constraints. Consequently, the optimization problem \eqref{eq:mtns:CSIP_original} consists of a \emph{finite} set of decision variables and \emph{infinitely} many of constraints each parameterized by \(\sip \Let (\uvar,\uvarr) \in \uset \times \Rbb^n\). We write the constraint \(v^{\top}\dvm v \ge 0\) for all \(v\in \Rbb^n\) as:
\begin{equation}
    \uvarr\transpose\dvm \uvarr \geq 0 \quad\text{for all }v \in K \subset \Rbb^n,
\end{equation}
where \(K \Let \Ball_{2}^n[0,1]\) is a Euclidean unit closed ball in \(\Rbb^n\). These types of reduction techniques are well studied in \emph{trust-region methods} (see \cite{ref:ARC-NIMG-PLT-00} for a detailed exposition) and commonly employed in various algorithms for tractability. The reformulated optimization problem, for which we will establish a tractable algorithm to obtain its near-optimal solution:
\begin{equation}
	\label{eq:CSIP_mod1}
	\begin{aligned}
		\ol{\ov}^{\ast}=& \inf_{\dvm \in \dvset} && \inprod{C}{\dvm} \\
		& \sbjto && \begin{cases}
			 \inprod{\uvar}{\dvm} \leqslant b \quad \text{for all}\, \uvar\in \uset,\\
            \uvarr\transpose\dvm \uvarr \geq 0 \quad \text{for all }\uvarr \in K \subset \Rbb^n;
            % \dvm  \in \dvset,	
            \end{cases}
	\end{aligned}
\end{equation}
naturally, \(\ol{\ov}^{\ast} = y^{\ast}\). 
We write \(C = (C_1\; C_2 \; \cdots \; C_n)\), \(\uvar = (\uvar_1\; \uvar_2 \; \cdots \; \uvar_n)\), and \(\dvm = (\dvm_1\; \dvm_2 \; \cdots \; \dvm_n)\) where \(C_i,\,\uvar_i,\dvm_i \in \Rbb^n\) are the columns of these matrices. By vectorizing our formulation and extracting the columns of \(A= (A_1,\ldots,A_n) \in \mathcal{A}\), the optimization problem \eqref{eq:mtns:CSIP_original} can be further simplified as `\(\sisdp\)' of the following form:
\begin{equation}
\label{eq:main SD-SIP}
    \begin{aligned}
    \ol{\ov}^{\ast} = &\inf_{(\dvm_i)_{i=1}^n } && \sum_{i=1}^n C_i{\transpose} \dvm_i \\
		& \sbjto && \begin{cases}
			 \sum_{i=1}^n \uvar_i{\transpose} \dvm_i \leqslant b \\ 
             \text{for all }A = (A_1,\ldots,A_n)\in \mathcal{A} \text{ where }i=1,\ldots,n,\\
            \uvarr{\transpose}\dvm \uvarr \geq 0 \quad \text{for all }\uvarr \in K \subset \Rbb^n.
            \end{cases}
    \end{aligned}
\end{equation}
\begin{remark}
    \label{rem:on problem SD-SIP}
Notice that the optimization problem~\eqref{eq:main SD-SIP} is a convex semi-infinite program with the constraint index set given by \(\mathcal{A} \times K\), which in its current form, is an NP-hard problem. To establish a computationally tractable approach to solve \eqref{eq:main SD-SIP} we take the route given in \cite{ref:DasAraCheCha-22}. To this end, we first define and establish certain structural properties of the function \(\gfunc\) in subsequent sections.
\end{remark}
\subsection{Preliminary results}
The chief contribution in \cite{ref:DasAraCheCha-22} was the translation of a semi-infinite program to a relaxed convex program with finite constraints where the constraints of the latter are selected in an intelligent manner and a global optimization is solved as an intermediary step. We adopt this technique to establish a method to directly tackle \(\sisdp\): Let us define the augmented \emph{semi-infinite variable} by \(\sParam \Let \bigl((\uvar_1 \; \uvar_2 \; \ldots \; \uvar_n), \uvarr\bigr) = \bigl(\uvar,\uvarr \bigr)\) taking values in \(\mathcal{A} \times K\). Let us further denote by 
\[
    \ol{N} \Let \frac{n(n+1)}{2},
\]
the dimension of the decision space.
Then \(\sParam_1, \sParam_2,\ldots,\sParam_{\ol{N}}\) corresponds to an \( \ol{N}\mbox{-}\)tuple in 
\[\totconset \Let (\mathcal{A} \times K)^{\ol{N}}\] 
where \(\sParam_j =  \bigl((\uvar^j_1 \; \uvar^j_2 \; \ldots \; \uvar^j_n), \uvarr_j\bigr) = \bigl(\uvar^j,\uvarr_j \bigr)\) for each \(j=1,2,\ldots, \ol{N}\).  Define \(\ol{\sParam} \Let (d_1,d_2, \dots,\sParam_{\ol{N}})\), and following \cite[\S2]{ref:DasAraCheCha-22}, we define the relaxed feasibility set by
\begin{equation}\label{eq:si-sdp feas}
     \rfeas(\ol{\sParam})  \Let  \left\{(\dvm_i)_{i=1}^n \;\middle\vert\;  
    \begin{array}{@{}l@{}}
        \sum_{i=1}^n \bigl(\uvar_{i}^j\bigr){\transpose}\dvm_{i}^j \leq b,\, \uvarr_j{\transpose}\dvm \uvarr_j \geq 0 \\
        \text{ for }\sParam_1,\ldots,\sParam_{\ol{N}} \in \totconset
        \text{ where } \\ \sParam_j= \bigl((\uvar^j_1 \; \uvar^j_2 \; \ldots \; \uvar^j_n), \uvarr_j\bigr)
        \end{array}
        \right\}.
\end{equation}
We define the function \(\gfunc:\totconset \lra \Rbb\) by 
\begin{equation}
    \label{eq:minimization problem}
	   \ol{\sParam} \mapsto  \gfunc(\ol{\sParam}) \Let   \inf \aset[\bigg]{\sum_{i=1}^n C_i{\transpose} \dvm_i \in \R \suchthat (\dvm_i)_{i=1}^n \in \rfeas(\ol{\sParam})}. 
\end{equation}
The following technical assumption aids in proving the main result ---Theorem \ref{thrm:mtns:value_func_equality} --- of this section.
\begin{assumption}
\label{assum:slater's}
We stipulate for the problem \eqref{eq:main SD-SIP} that the following Slater-type condition holds: there exists \(\dvar\mbox{-}\)tuple \((d_1,d_2,\ldots,d_{\dvar})\) such that the feasible set \eqref{eq:si-sdp feas} is nonempty. 
\end{assumption}
\begin{proposition}
    \label{prop:regularity of G}
    Consider the optimization problem \eqref{eq:mtns:CSIP_original} along with its associated data and let Assumption \ref{assum:slater's} hold. Then \(
    \gfunc(\cdot)
    \) admits the following properties:
    \begin{enumerate}[label=\textup{(\ref{prop:regularity of G}-\alph*)}, leftmargin=*, widest=b, align=left]
        \item \label{it:convexity} If \(\uset\) is a convex, then the function \(\gfunc:\totconset \lra \Rbb\) defined in \eqref{eq:minimization problem} is convex;

        \item \label{it:usc} The function \(\gfunc:\totconset \lra \Rbb\) is upper semicontinuous.
    \end{enumerate}
\end{proposition}
\begin{proof}
With the convexity of \(\uset\), the convexity of \(\gfunc(\cdot)\) follows readily from \cite[Proposition 1]{ref:DasAraCheCha-22}. Observe that in \eqref{eq:si-sdp feas}, the map \((\dvm_1,\ldots,\dvm_n) \mapsto \sum_{i=1}^n (\uvar_i)^{\top}\dvm_i\) is continuous for all \(j=1,\ldots,\ol{N}\) and \(X \mapsto \uvarr_j^{\top}\dvm \uvarr\) is also continuous. Fix \((\ol{\dvm_i})_{i=1}^n \subset \rfeas(\ol{\sParam})\). From Assumption \ref{assum:slater's} and the continuity of the constraint functions \((\dvm_1,\ldots,\dvm_n) \mapsto \sum_{i=1}^n (\uvar_i)^{\top}\dvm_i\) and \(X \mapsto \uvarr_j^{\top}\dvm \uvarr\), there exist a sequence \((\dvm_{i}^m)_{m \in \N}\subset \rfeas(\ol{\sParam})\) such that \(\dvm^m_i \lra \ol{\dvm_i}\) for each \(i=1,2,\ldots,n\) as \(m \lra + \infty\) \cite{ref:GanchaLCSS}. This implies that the constraint qualification condition (CQ) in \cite[Definition 5.3, p. 53]{ref:param_opt_still} holds. We get the upper semicontinuity of \(\gfunc(\cdot)\) invoking \cite[Lemma 5.4-(b), p. 54]{ref:param_opt_still}.  
\end{proof}
Next we state the main result of this section.

\begin{theorem}\label{thrm:mtns:value_func_equality}
    Consider the optimization problem \eqref{eq:main SD-SIP} and suppose that Assumption \eqref{assum:slater's} is in force. Consider also the convex SIP \eqref{eq:minimization problem} along with its associated data and notations. Consider the global maximization problem 
    \begin{equation}
    \label{e:global_max_prob}
    \begin{aligned}
        & \sup_{\ol{d}\in \totconset}
        &&\gfunc(\ol{d}).
    \end{aligned}
    \end{equation}
    Then there exists  \(\ol{d}\as \in \totconset\) that solves \eqref{e:global_max_prob}. Moreover, \(\ol{y}\as=\gfunc(\ol{d}\as)\).
    \end{theorem}
\begin{proof}
The existence of an optimizer \(\ol{\sParam}^{\ast}\) follows immediately from the upper semicontinuity of \(\gfunc(\cdot)\) in the assertion Proposition \ref{it:usc}, the compactness of \(\totconset\), and from the Weierstrass Theorem \cite[Theorem 2.2]{ref:OG10}. Hence the first assertion stands established. We observe that the set \(\rfeas(\ol{\sParam})\) is a nonempty (follows from Assumption \ref{assum:slater's}) closed and convex set. Invoking Proposition \cite[Theorem 1]{ref:DasAraCheCha-22}, we assert \(\ol{y}\as=\gfunc(\ol{d}\as)\).  
\end{proof}

\subsection{Algorithm to solve \(\sisdp\)}
Here we provide Algorithm \ref{alg:mtns:sisdp} to solve \eqref{eq:main SD-SIP}. Let us fix some notations that are used in the algorithm: At \(k^{\text{th}}\) iteration, \(\ol{\sParam}^k \in \mathcal{A} \times K \) denotes the optimal solution  of the global maximization problem \eqref{eq:main SD-SIP}. Here \(\ol{\sParam}^k \Let \big( \sParam_1^k,\ldots,\sParam_{\ol{N}}^k\bigr)\) and \(\sParam_j^k = \bigl((\uvar^{j,k}_1 \; \uvar^{j,k}_2 \; \ldots \; \uvar^{j,k}_n), \uvarr_j\bigr)\) for each \(j=1,\ldots,\ol{N}\), and we define by \(s^k = \bigl(\dvm^k_i \bigr)_{i=1}^n\) the solution to \eqref{eq:minimization problem}.
%%%%%%%%%%%%%%%%%%%%%%%%%%%%%%%%%%%%
%%%%%% Algorithm for SI-SDP %%%%%%%%
%%%%%%%%%%%%%%%%%%%%%%%%%%%%%%%%%%%%
\begin{algorithm}
	\SetAlgoLined
  	\DontPrintSemicolon
	\SetKwInOut{ini}{Initialize}
    \SetKwInOut{giv}{Data}
    \giv{Stopping criterion $\SC(\cdot)$, threshold for the stopping criterion \(\tau\);}
    \ini{initialize \(\ol{\sParam^0}\) and \(s^0\), initial guess for maximum value \(\gfunc_{\max}\), initial guess for initial solution \(\ol{s}\);}
	\BlankLine
	\While{$\SC(k) \leqslant \tau$}{
		\emph{Sample}: \(\ol{\sParam}^k \in \mathcal{A} \times K \) \;
		\emph{Evaluate} \(\gfunc^k \Let \gfunc(\ol{\sParam}^k)\) as defined in \eqref{eq:minimization problem} \;
		\emph{Recover} the solution  \(s^k \) by solving the minimization problem \eqref{eq:minimization problem}
        \;
		\uIf{$\gfunc^k \geqslant \gfunc_{\max}$}{Set $\gfunc_{\max} \gets \gfunc^k$\;
        Set ${\ol{s}} \leftarrow s^k$}
        Update $k \gets k+1$ \;
	}
\caption{Algorithm to solve \eqref{eq:main SD-SIP}}\label{alg:mtns:sisdp}
\end{algorithm}

\section{\(\sisdp\text{-}\)based algorithm to solve distributionally robust MPC}\label{sec:main_results}     
This section develops a framework to solve distributionally robust model predictive control (DRMPC) problems based on the \(\sisdp\) algorithm developed in \S\ref{sec:prob_form} with applications of mathematical finance in mind. We first set up a DRMPC problem and subsequently we will adapt Algorithm \ref{alg:mtns:sabaro} to find its solution online. 

\subsection{DRMPC formulation}
Let \(d, m, q \in \N\) and consider a time-invariant discrete-time stochastic control system given by the recursion
\begin{equation}
    \label{eq:mtns:system}
    \st_{t+1} = Ax_t+Bu_t+\sum_{j=1}^q \bigl(C_j\st_t+D_ju_t \bigr)\dist_t^j,\,\, \st_0=\xz,\,\, t \in \N,
\end{equation}
where \(\st_{t} \in \Rbb^{d}\) and \(\ut_t \in \Rbb^m\) are the vectors representing the states, control inputs, and \(\dist_t^j \in \Rbb\) for \(j=1,\ldots,q\) are i.i.d. (independent and identically distributed)  random uncertainty variables at time \(t\) that multiply the state and the control vectors. The matrices \(A \in \Rbb^{d \times d}\) and \(B \in \Rbb^{d \times m}\) are respectively, the system and the control matrices, and \(C_j \in \Rbb^{d \times d}\), \(D_j \in \Rbb^{d \times m}\) for \(j=1,\ldots,q\). We assume that a perfect measurement of the state \(x_t\) is available. The vector \(\dist_t^j\) is a \(\Rbb\)-valued random \emph{process noise} with possibly unbounded support such that \(\EE[\dist_t^j]=0\) and \(\EE[\dist_t^j \dist_t^l]= 0\) for \(j \neq l\). % and \(\EE[\dist_t^j\dist_t^j] = 1\). 
Compactly, we write \(\dist_t = (\dist_t^1 \, \dist_t^2 \,\cdots\, \dist_t^q)\) as a \(q\mbox{-}\)dimensional vector.

We assume that the following objects are given:
\vspace{1mm}
\begin{enumerate}[label=\textup{(\ref{eq:mtns:system}-\alph*)}, leftmargin=*, widest=b, align=left]
\item \label{eq:mtns:system:costs} A time horizon \(\horizon \in \N\); a quadratic in state-action \emph{cost-per-stage} function
\((\dummyx,\dummyu) \mapsto \cost(\dummyx,\dummyu) \Let \inprod{\dummyx}{Q\dummyx}+ \inprod{\dummyu}{R\dummyu}\) and a quadratic \emph{final-stage cost} function \(\dummyx \mapsto \fcost(\dummyx) \Let \inprod{\dummyx}{P \dummyx}\) with \(Q \in \psdmat{d}\), \(P \in \psdmat{d}\) and \(R \in\pdmat{m}\).

\item Design parameters \(\beta_c>0\) and \(\beta_c^X\) such that for the state-action concatenated vector \(h_t \Let (x_t \,\,u_t) \in \Rbb^{d+m}\); we impose the following set of constraints on the state and control variables: let \(H_l \succeq 0\), \(H^X_l \succeq 0\), \(f_l \in \Rbb^{d+m}\), \(f^{\st}_l \in \Rbb^d\), and \(t=0,\ldots,\horizon-1\)
\begin{equation}\label{eq:chance_constraints}
\begin{aligned}
\mathcal{C}_{\dummyx,\dummyu}(\xz) \Let\begin{cases}
  \EE_{\xz} \bigl[ h_t^{\top}H_l h_t + f_l^{\top} h_t \bigr] \le \beta_c \text{ with }l=1,\ldots,L_1, \\
\EE_{\xz} \bigl[ x_t^{\top}H^{\st}_l x_t + (f^{\st}_l)^{\top} \st_t \bigr] \le \beta_c^{\st} \text{ with }l=1,\ldots,L_2.
\end{cases}
\end{aligned}
\end{equation}
\end{enumerate}
\vspace{2mm}
To control and subsequently stabilize the dynamics \eqref{eq:mtns:system} by applying a closed-loop control law we solve a constrained stochastic optimization problem with an expected cost in a receding horizon manner. In general, numerical tractability is always an issue for such problems where the minimization is performed over control policies instead of control sequences, which is common in conventional deterministic formulations. Thus, for tractability we consider control laws of the following form \cite{ref:JAP:CS:SMPC:TAC:2010}: let \(\mean_{\st_t} \Let \EE_{\xz}\expecof[]{\st_t}\) and define
\begin{align}\label{eq:st_fbpar}
    u_t = \uz_t+K_t(x_t - \mean_{\st_t}).
\end{align}
The input and the gain sequence \(\uz_t\) and \(K_t\) respectively, in \eqref{eq:st_fbpar} are chosen via solving an appropriate optimization problem. Define 
\begin{align}\label{eq:tot_cost}
\hspace{-2mm}\cst_{\horizon}\bigl(\xz,u_t,\dist_t\bigr) \Let \inprod{x_N}{S \dummyx_N}+\sum_{t=0}^{\horizon-1} \inprod{\st_t}{Q\st_t}+ \inprod{u_t}{R u_t}.
\end{align}
Define \(\upsilon_t \Let \aset[\big]{(\bar{u}_t,K_t)_{t=0}^{\horizon-1}}\). Given the above ingredients, the baseline stochastic optimal control problem is given by:
\begin{equation}
\label{eq:baseline_SOCP}
\begin{aligned}
		& \hspace{-2mm}\min_{\upsilon_t} && \hspace{-2mm}\EE_{\xz} \expecof[\big]{\cst_{\horizon}\bigl(\xz,\cont_t,\dist_t\bigr)}\\
		& \hspace{-2mm}\sbjto && \hspace{-4mm}\begin{cases}
			 \text{dynamics }\eqref{eq:mtns:system}\text{ for each }t = 0,\ldots,\horizon-1,\\\st_0 = \xz,\,\text{the constraints } \eqref{eq:chance_constraints},\,\EE_{\xz}\expecof[\big]{\st_{\horizon}^{\top}\admfinst \st_{\horizon}} \le \alpha,\\ \text{the control parameterization } \eqref{eq:st_fbpar} \text{ with }u_0 = \ol{u}_0.
		\end{cases}
	\end{aligned}
\end{equation}
The above stochastic optimal control is quite well known from the finance viewpoint and this setting has been adopted in several applications related to finance such as portfolio optimization \cite{ref:JAP:ACC:portfolio}, constrained index tracking \cite{ref:JAP:SCH:index_tracking}, trading applications \cite{ref:SM:LJH:JAP:finance} etc. 

\par Let \(\pset\) be the set of probability measures defined on \((\Rbb^q,\Borelsigalg(\Rbb^q))\) where \(\Borelsigalg(\Rbb^q)\) is the Borel sigma-algebra defined on \(\Rbb^q\) in standard fashion. Let \(\hat{\var}_{\dist} \in \pdmat{d}\) and \(\mean_{\dist}\) denote the empirical variance and empirical mean of the process noise \(\dist\) and are mathematically written as
\begin{align*}
    \hat{\var}_{\dist} = \frac{1}{\nsamp}\sum_{i=1}^{\nsamp} \xi_i \xi_i^{\top},
\end{align*}
with \(\xi_i \sim \PP_{\dist}\) are \(d\mbox{-}\)dimensional i.i.d. samples. Define the ambiguity set by 
\begin{equation}\label{eq:ambg_set}
     \amgset_{\nsamp}  \Let  \left\{\PP_{\dist} \in \pset \;\middle\vert\;  
    \begin{array}{@{}l@{}}
        \EE\cexpecof[\big]{\dist_t} = 0,\,
        \EE \expecof[\big]{\dist_t \dist_t^{\top}} = \var_{\dist} \preceq \pa \hat{\var}_{\dist}
        \end{array}
        \right\}.
\end{equation}
Note that the covariance matrix \(\var_{\dist}\) is diagonal and the elements correspond to the variances of the coordinates of \(\dist_t\), i.e., \(\EE \expecof[\big]{\dist_t^j \dist_t^j} = \va_{\dist,j}\) for \(j = 1,2,\ldots, q\). Consequently, the empirical covariance is written as \[\Rbb^{q \times q} \ni \hat{\var}_{\dist} = \diag (\hat{\va}_{\dist,1} \; \hat{\va}_{\dist,2} \; \cdots \; \hat{\va}_{\dist,q}),\] 
and the distributionally robust version of the stochastic optimal control problem \eqref{eq:baseline_SOCP} is given by: 
\begin{equation}
	\label{eq:DRMPC}
	\begin{aligned}
		& \min_{\upsilon_t} \max_{\PP \in \amgset_{\nsamp}} &&  \EE_{\xz} \expecof[\big]{\cst_{\horizon}\bigl(\xz,u_t,\dist_t\bigr)}\\
		& \sbjto && 
			\text{constraints as specified in }\eqref{eq:baseline_SOCP}.
	\end{aligned}
\end{equation}
%%%%%%%%%%%%%%%%%%%%%%%%%%%%%%%%%%%%%%
%%%%%%% SI-SDP formulation %%%%%%%%%%%
%%%%%%%%%%%%%%%%%%%%%%%%%%%%%%%%%%%%%%
\subsection{\(\sisdp\) formulation of \eqref{eq:DRMPC}}
DRO problems can be directly solved via employing techniques from semi-infinite optimization \cite{FL:SM:DRO_review19,ref:CheAliBanHota}. With this motivation, we reformulate the optimization problem \eqref{eq:DRMPC} as a \(\sisdp\) of the form \eqref{eq:main SD-SIP}. We take the same route as in \cite[\S III]{ref:JAP:CS:SMPC:TAC:2010} but instead we add a distributional uncertainty in our formulation and allow it to take values from an uncountable set, which makes the SDP a semi-infinite program. Define the quantities \(\acl_t \Let A+ BK_t\), \(\cbar_{j,t} \Let C_j \var^{\st}_t + D_j K_t \var^{\st}_t\), and \(\dbar_{j,t} \Let C_j \mean_{\st_t} + D_j \uz_t\). The mean \(\mean_t\) and the covariance \(\var^{\st}_t\) of \eqref{eq:mtns:system} are computed as 
\begin{align}\label{eq:mean_dyn}
    \mean_{\st_{t+1}} = A \mean_{\st_t} + B \uz_t ,
\end{align}
and 
\begin{align}
    \label{eq:st_cov}
    \var^{\st}_{t+1} \Let \EE_{\xz} \expecof[]{\st_{t+1}}  = \acl_t\var^{\st}_t \acl_t^{\top} &+ \sum_{i=1}^{q} \va_{\dist,j}\cbar_{j,t}\bigl(\var^{\st}_t\bigr){\inverse} \cbar_{j,t}^{\top} + \sum_{i=1}^{q} \va_{\dist,j}\dbar_{j,t} \dbar_{j,t}^{\top}. 
\end{align}
Define \(U_t \Let K_t \var^{\st}_t\). We stipulate that \(\var^{\st}_t \in \pdmat{d}\) for \(t=1,\ldots,N,\) and using the Schur complement lemma \cite{ref:Horn:Johnson:MA}, the covariance \eqref{eq:st_cov} is equivalent to
\begin{align}
    \label{eq:st_cov_m1}
    \begin{pmatrix}
        \var^{\st}_{t+1} & * & * & * \vspace{2mm}\\
        \acl_t^{\top} & \var^{\st}_t & 0 & 0 \\
        \sum_{j=1}^q \va_{w,j}^{1/2} \cbar_{j,t} & 0 & \var^{\st}_t & 0\\
        \sum_{j=1}^q \va_{w,j}^{1/2} \dbar_{j,t} & 0 & 0 & \var^{\st}_t
    \end{pmatrix} \succ 0 \;  
\end{align}
for \(t=1,\ldots,\horizon-1\), and the initial condition,
\begin{align}
     \label{eq:st_cov_m2}
     \begin{pmatrix}
        \var^{\st}_{1} & * \vspace{2mm}\\
         \va_{w,j}^{1/2} \dbar_{j,0} & \identity_d 
    \end{pmatrix} \succ 0 \; \text{for }t=0. 
\end{align}
The rest of the constraints can be equivalently transformed into linear matrix inequalities: Define the matrix variable
\begin{align*}
    P_t \Let \begin{pmatrix}
        P^{\st}_t & P^{\st \cont}_t\\
        \bigl(P^{\st \cont}_t\bigr)^{\top} & P^{\cont}_t
    \end{pmatrix} \; \text{for }t= 0,\ldots,N-1.
\end{align*}
Then for \(t = 0,\ldots,N-1,\) 
\begin{align}
    \label{eq:const_1}
   \hspace{-2mm} \begin{pmatrix}
        P_t & \begin{pmatrix}
            \var^{\st}_t & U_t^{\top}\\
            \mean_{\st_t}^{\top} & \uz_t^{\top}
        \end{pmatrix}^{\top}\\
        \begin{pmatrix}
            \var^{\st}_t & U_t^{\top}\\
            \mean_{\st_t}^{\top} & \uz_t^{\top}
        \end{pmatrix} & \begin{pmatrix}
            \var^{\st}_t & 0\\
            0 & 1
            \end{pmatrix}
    \end{pmatrix} \succeq 0\,\text{with }  \begin{pmatrix}
        P^{\st}_N - \var^{\st}_N & \mean_{\st_N}\\
        \mean_{\st_N}^{\top} & 1
    \end{pmatrix} \succeq 0.
\end{align}
The state-action constraints, the state-only constraints, and the terminal constraint in \eqref{eq:chance_constraints} are written as
\begin{align}\label{eq:const_2}
    \trace(H_{l} P_t) + f_l^{\top} \begin{pmatrix}
        \mean_{\st_t}\\\uz_t
    \end{pmatrix} \leq \beta_c,
\end{align}
for \(t = 1\ldots ,N-1\) with \(l = 1,\ldots,L_1\), and
\begin{align}\label{eq:const_3}
    \trace(H^{\st}_{l}P^{\st}_t) + \bigl(f^{\st}_l  \bigr)^{\top} \mean_{\st_t} \leq \beta_c^{\st},
\end{align}
for \(t=1,\ldots,\horizon\) with \(l=1,\ldots,L_2\), and \(\trace(\admfinst P_N) \leq \alpha\).
\begin{assumption}
    \label{assum:ab_st}
    We stipulate that the set
\begin{equation*}\label{eq:parameter_set}
      \left\{\hat{\var}_{\dist} \;\middle\vert\;  
    \begin{array}{@{}l@{}}
        \hat{\var}_{\dist} \succeq 0, \text{ for all }\var_{\dist} \preceq \pa \hat{\var}_{\dist} \text{ there}\\\text{exists }\PP_{\dist} \in \pset \text{ such that }\EE_{\PP_{\dist}} \expecof[\big]{\dist_t \dist_t^{\top}} = \var_{\dist}
        \end{array}
        \right\} \neq \emptyset.
\end{equation*}
\end{assumption}
Define the augmented variable \(\eta_t \Let \bigl(\uz_t,\mean_{\st_t}, P_t, \var^{\st}_{t},U_t\bigr)\) and the \emph{ambiguity parameter set} by
\begin{align}
    \amgball \Let \aset[\big]{\var_{\dist} \suchthat \var_{\dist} \succ 0,\var_{\dist} \preceq \gamma \hat{\var}_{\dist}}.
\end{align}

With Assumption \ref{assum:ab_st} in place \(\sisdp\)  version of \eqref{eq:DRMPC} is:
\begin{equation}
	\label{eq:SDSIP_DRMPC}
	\begin{aligned}
		& \min_{\eta_t} \max_{\var_w \in \amgball} &&  \sum_{t=0}^{\horizon-1}\trace \bigl(MP_t\bigr)+ \trace \bigl(S P^{\st}_{\horizon}\bigr) \\
		& \sbjto && 
			\text{constraints }\eqref{eq:st_cov_m1}-\eqref{eq:const_2}, \,\trace(\admfinst P_{\horizon}) \leq \alpha.
	\end{aligned}
\end{equation}
We denote the value function of the optimal control problem \eqref{eq:SDSIP_DRMPC} by \(\cst_{\horizon}\as(\cdot)\), which is a mapping from the set of all feasible initial states \(X_{\horizon}\) to real numbers.

Note that the \(\sisdp\) \eqref{eq:SDSIP_DRMPC} can be translated to a minimization problem from the min-max realization by adding a slack variable \(r_0 \in \lcro{0}{+\infty}\) without changing the value of the ensuing mathematical program. This generates an optimization problem of the form \eqref{eq:main SD-SIP} allowing us to directly apply all the machinery developed in \S \ref{sec:prob_form}. Let \(\ol{N} \Let \dim(\eta_t)+1\). Consider the relaxed optimization problem
\begin{align}
    \label{eq:relaxed_DRMPC}
    \ol{\gfunc}(\var^1_{\dist},\ldots, \var^{\ol{N}}_{\dist};\xz)  =&  \inf_{r_0, (\eta_t)_{t=0}^{\horizon-1}} && r_0 \\
		& \sbjto && 
			\hspace{-3mm}\begin{cases}
            \sum_{t=0}^{\horizon-1}\trace \bigl(MP_t\bigr)+ \trace \bigl(S P^{\st}_{\horizon}\bigr) \leq r_0 \\ \text{for all }\var^1_{\dist},\ldots,\var^{\ol{\horizon}}_{\dist} \in \amgball,\\
            \text{constraints }\eqref{eq:st_cov_m1}-\eqref{eq:const_2}, \, \\
            r_0 \in \lcro{0}{+\infty},\,\trace(\admfinst P_{\horizon}) \leq \alpha.\nn
            \end{cases}
\end{align}
\begin{corollary}
    \label{cor:value_func_equality}
    Consider the OCP \eqref{eq:SDSIP_DRMPC} and suppose that Assumption \ref{assum:slater's} and Assumption \ref{assum:ab_st} are in force. Fix an \(\xz \in X_{\horizon}\). Consider the global maximization problem
    \begin{align}
        \label{eq:g_func}
\sup_{(\var^1_{\dist},\ldots,\var^{\ol{\horizon}}_{\dist}) \in \amgball^{\ol{\horizon}}}
        \ol{\gfunc}(\var^1_{\dist},\ldots,\var^{\ol{\horizon}}_{\dist};\xz).
    \end{align}
    Then there exists \(\bigl(\var^1_{\dist},\ldots, \var^{\ol{N}}_{\dist}\bigr)\) that solves \eqref{e:global_max_prob}, and we have \(\cst_{\horizon}\as(\xz)=\ol{\gfunc}(\var^1_{\dist},\ldots, \var^{\ol{N}}_{\dist};\xz)\).
\end{corollary}
\begin{proof}
    A proof of Corollary \ref{cor:value_func_equality} follows immediately from Theorem \ref{thrm:mtns:value_func_equality}. 
\end{proof}

%%%%%%%%%%%%%%%%%%%%%%%%
%%%%%% Algorithm %%%%%%%
%%%%%%%%%%%%%%%%%%%%%%%%
\subsection{Algorithm to solve \eqref{eq:g_func}}\label{subsec:alg_drpmc} 
 Algorithm \ref{alg:mtns:sabaro} solves the global maximization problem \eqref{eq:g_func} which is a crucial step in solving \eqref{eq:SDSIP_DRMPC}. We define some notations that will be useful in Algorithm \ref{alg:mtns:sabaro}: 
 \(\ol{\var} \Let \bigl(\var^1_{\dist},\ldots, \var^{\ol{N}}_{\dist} \bigr)\)
 and \(\beta_t \Let \bigl(r_0,(\eta_{t})_{t=0}^{\horizon-1}\bigr)\). \footnote{Throughout the algorithm, all superscripts correspond to the iteration of the global optimization routine.}

\begin{algorithm}
	\SetAlgoLined
  	\DontPrintSemicolon
	\SetKwInOut{ini}{Initialize}
    \SetKwInOut{giv}{Data}
    \giv{Stopping criterion $\SC(\cdot)$, threshold for the stopping criterion \(\tau\) and fix \(\xz \in X_{\horizon}\);}
    \ini{initialize constraint indices  \(\ol{\var}^0 \Let \Bigl(\var^{1,0}_{\dist},\ldots, \var^{\ol{N},0}_{\dist} \Bigr) \in  \amgball^{\ol{\horizon}}\)
 and the solution \(\beta^0_t \Let \bigl(r^0_0,(\eta^{0}_{t})_{t=0}^{{\horizon}-1}\bigr)\) to \eqref{eq:relaxed_DRMPC}, initial guess for maximum value \(\ol{\gfunc}_{\max}\), initial guess for the initial solution \(\ol{\beta}\);}
	\BlankLine
	\While{$\SC(k) \leqslant \tau$}{
		\emph{Sample}: \(\ol{\var}^k  \in \amgball^{\ol{\horizon}}\)\;
		\emph{Evaluate} \(\ol{\gfunc}^k \Let \ol{\gfunc}(\ol{\var}^{k};\xz)\) as defined in \eqref{eq:g_func} \;
		\emph{Recover} the solution  \(\beta^k \in \argmin_{\bigl(r^k_0,(\eta^k_t)_{t=0}^{{N}-1}\bigr)} \aset[\big]{r_0 \,\big|\,\text{constraints in } \eqref{eq:relaxed_DRMPC} \text{ hold at }\ol{\var}^{k}} \)
        \;
		\uIf{$\ol{\gfunc}^k \geqslant \ol{\gfunc}_{\max}$}{Set $\ol{\gfunc}_{\max} \gets \ol{\gfunc}^k$\;
        Set ${\ol{\beta}} \leftarrow \beta^k$}
        Update $k \gets k+1$ \;
	}
\caption{Algorithm to solve \eqref{eq:g_func}}\label{alg:mtns:sabaro}
\end{algorithm}

\begin{remark}\label{rem:comparision_alg_1_and_alg_2}
(On Algorithm \ref{alg:mtns:sisdp} and Algorithm \ref{alg:mtns:sabaro})  Algorithm \ref{alg:mtns:sabaro} is a specialized version of the general Algorithm \ref{alg:mtns:sisdp} established in \S\ref{sec:prob_form} in the context of the DRPMC problem. For the original \(\sisdp\) problem \eqref{eq:main SD-SIP} the semi-infinite parameters are \(\sParam \Let \bigl((\uvar_1 \; \uvar_2 \; \ldots \; \uvar_n), \uvarr\bigr) = \bigl(\uvar,\uvarr \bigr)\), whereas for the DRMPC problem \eqref{eq:SDSIP_DRMPC}, \(\var_w \) is the semi-infinite parameter. Subsequently, the relaxed inner-minimization problem \eqref{eq:minimization problem} was a function of \(\ol{\sParam} \Let (d_1,d_2, \dots,\sParam_{\ol{N}})\) \emph{finite samples} of the parameters; similarly, for the problem \eqref{eq:relaxed_DRMPC}, \(\ol{\gfunc}(\cdot)\) is a function of \((\var^1_{\dist},\ldots, \var^{\ol{N}}_{\dist})\) samples. Finally, an appropriate global maximization is solved (\eqref{e:global_max_prob} for the general \(\sisdp\) problem and \eqref{eq:g_func} for the DRMPC problem \eqref{eq:SDSIP_DRMPC}) to find the optimizers.
\end{remark}
\section{Numerical experiment}
We consider the following second-order discrete-time controlled dynamical system~\cite{ref:JAP:CS:SMPC:TAC:2010} with an i.i.d. process noise \(w_t \sim \mathcal{N}(0,1)\) that multiplies both \(\st_t\) and \(u_t\):
\begin{align}\label{num:mtns:sys_dyn}
    x_{t+1} = Ax_t + Bu_t + \bigl( C x_t + D u_t\bigr)w_t,
\end{align}
where the corresponding matrices in \eqref{num:mtns:sys_dyn} are given as: 
\begin{align}
    A \Let \begin{pmatrix}
        1.02 & -0.1 \\ 0.1 & 0.98
    \end{pmatrix}, \quad B \Let \begin{pmatrix}
        0.10 & 0 \\ 0.05 & 0.01
    \end{pmatrix}, \nn  \\ 
    C \Let \begin{pmatrix}
        0.04 & 0 \\ 0 & 0.04
    \end{pmatrix}, \quad D\Let \begin{pmatrix}
        0.04 & 0 \\ -0.04 & 0.008
    \end{pmatrix}. \nn
\end{align}
The state, the control weighting matrices, and the matrix \(\admfinst\) are chosen as 
\begin{align}
     Q \Let \begin{pmatrix}
        2 & 0 \\ 0 & 1
    \end{pmatrix}, \,\, R \Let \begin{pmatrix}
        5 & 0 \\ 0 & 20
    \end{pmatrix}, \,\,\admfinst \Let \begin{pmatrix}
        41.0331 & -5.7929 \\ -5.7929 & 54.3889
    \end{pmatrix}. \nn
\end{align}
For simplicity, we enforce only state constraints 
\begin{equation}\label{eq:chance_const}
\begin{aligned}
\mathcal{C}_{\dummyx}(\xz) \Let
  \EE_{\xz} \bigl[ (-2\,1)^{\top}x_t\bigr] \le 2.3, 
\end{aligned}
\end{equation} 
and the terminal constraint is \(\EE_{\xz}[x_{\horizon}\admfinst x_{\horizon}] \le 45\). 

\begin{figure}[H]
    \centering
    \includegraphics[scale=0.6]{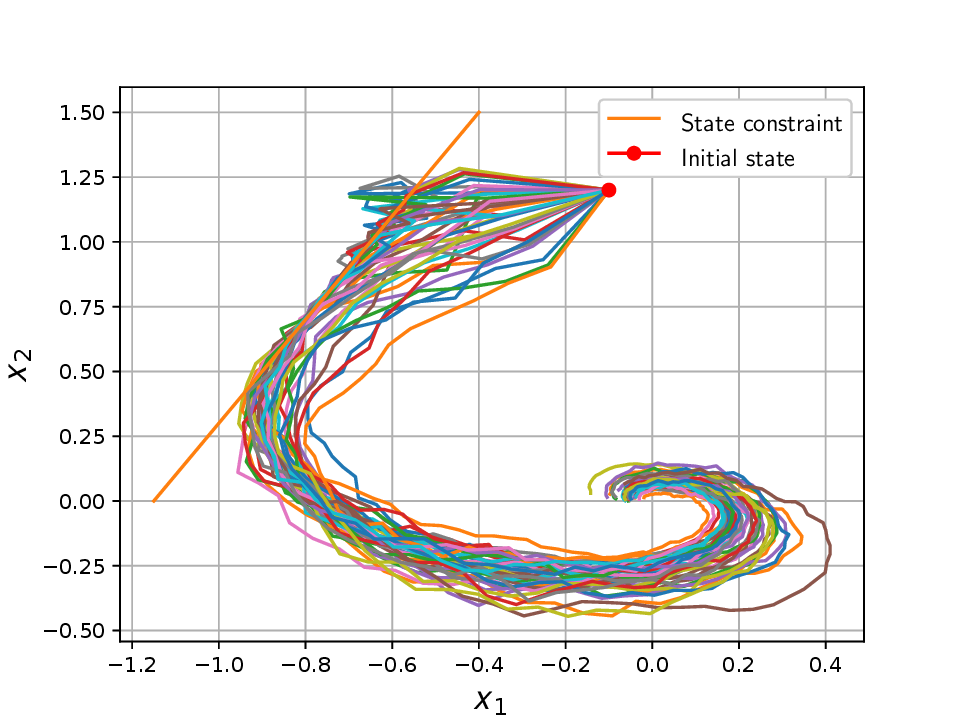}
    \caption{Phase portrait of the trajectories using our DRMPC Algorithm \ref{alg:mtns:sabaro}. It can be seen that the state trajectories are getting attracted towards the origin.}
    \label{fig:stable_effect}
\end{figure}
With these ingredients we considered the DRMPC problem of the form in \eqref{eq:DRMPC} which we reformulated and arrived at a problem of the form given in \eqref{eq:SDSIP_DRMPC} and subsequently employed Algorithm~\ref{alg:mtns:sabaro} to solve the ensuing \(\sisdp\). We performed our numerical experiment using Python 3.10 running on a \(36\) core server with Intel(R) Xeon(R) CPU E\(5-2699\) v\(3\), \(4.30\) GHz with \(32\) Gigabyte of RAM. We employed CVXPY with the MOSEK~\cite{ref:mosek} solver for solving semi-definite programs. To simulate multiple trajectories simultaneously we used the multiprocessing library. We generated 30 random samples using the statistics of process noise beforehand to emulate the availability of historical data for some practical application and use these samples to obtain a realistic estimate of the covariance \(\hat{\var}_{\dist} = 1.04\).To obtain a distributionally robust policy for the system~\eqref{num:mtns:sys_dyn}, we use the ambiguity set as defined in~\eqref{eq:ambg_set} using the estimate \(\hat{\var}_{\dist}\). With this data, \(40\) trajectories were generated starting from the initial state \(\ol{\st} = (0.1 \;1.2)^{\top}\). Figure \ref{fig:stable_effect} depicts the stabilizing characteristics of the DRMPC when initialized with \(\ol{\st} = (0.1 \;1.2)^{\top}\). 

\section{Concluding remarks}
This article established an algorithm for near-optimal solution of SI-SDPs and its application to DRMPC problems. The underlying discrete-time dynamical system contains uncertainties that multiply both the system state and control variables which are typically the governing dynamics in many finance applications. By introducing distributional uncertainty to a given SMPC we converted the ensuing DRMPC problem to an SI-SDP and applied Algorithm \ref{alg:mtns:sabaro} to cater to the DRMPC problem. This article reports our preliminary results on this front, in particular, stability and feasibility guarantees were not reported here. The immediate next step would be to develop the stability theory and establish an explicit synthesis algorithm for fast implementation; these results will be reported in our subsequent investigations.

\bibliographystyle{IEEEtran}
\bibliography{refs.bib}
\end{document}